\documentclass[11pt]{amsart}

\usepackage{mathtools}

\newtheorem{theorem}{Theorem}[section]
\newtheorem{proposition}[theorem]{Proposition}

{
\theoremstyle{definition}

}

\newcommand{\Auto}{{\rm Aut}}
\newcommand{\Endo}{{\rm End}}
\newcommand{\Img}{{\rm Im}}

\begin{document}

\title[Posets of width 4]{The Endomorphism Conjecture for Graded Posets with Whitney numbers at most 4}

\author{Mikl\'os B\'ona}
\address{Department of Mathematics, University of Florida, 358 Little Hall, PO Box 118105,
Gainesville, FL 32611-8105}
\email{bona@ufl.edu}
\thanks{The first author is partially supported by Simons Collaboration Award 940024.}

\author{Ryan R. Martin}
\address{Department of Mathematics, Iowa State University, 411 Morrill Road, Ames, IA, 50011-2104}
\email{rymartin@iastate.edu}
\thanks{The second author is partially supported by Simons Collaboration Award 709641.}

\begin{abstract} We prove the endomorphism conjecture for graded posets whose largest Whitney number is at most 4. In particular, this implies the endomorphism conjecture is true for graded posets of width at most 4. 
\end{abstract}


\begin{abstract} 
	We prove the endomorphism conjecture for graded posets with largest Whitney number at most 4.
\end{abstract}

\maketitle

\section{Introduction}
\hyphenation{Rut-kow-ski}
\subsection{The endomorphism conjecture}
 Let $P$ be any finite partially ordered set.  An {\em automorphism} of $P$ is an order-preserving bijection $f:P\rightarrow P$. That is,  if $a\leq b$, then  $f(a)\leq f(b)$.  An {\em endomorphism} of $P$ is a map  $g:P\rightarrow P$  which preserves order, but is not necessarily bijective. 
The automorphisms of $P$ form the group $\Auto(P)$, the endomorphisms of $P$ form the monoid $\Endo(P)$. Let $P_1,P_2,\ldots$ be an infinite sequence of distinct finite posets. 
The question as to whether
 \[\lim_{i\rightarrow \infty} \frac{|\Auto(P_i)|}{|\Endo(P_i)|}=0\] 
for {\em all} such sequences of posets was raised by I. Rival and A. Rutkowski~\cite{rival} and is still open. 
It is widely believed that the answer to this question is in the affirmative.
This is known as the {\em endomorphism conjecture}, though some authors call it the {\em automorphism conjecture}. 

\subsection{Earlier results} 
The endomorphism conjecture has been proved for lattices~\cite{liu-wan}, the direct products of certain posets~\cite{kuzjurin}, and  series-parallel posets \cite{schroder}. 
A recent addition to the literature is the paper~\cite{schroder-1} that proves rapid convergence for posets of 
dimension 2. 

In \cite{bona}, one of the present authors proved the endomorphism conjecture for sequences of finite posets with a minimum element for which there is a positive integer $k$
so that the number of elements of height $k$ goes to infinity. This implies, for instance, that the conjecture holds for posets with a 0 that have bounded height. 

Shortly after an earlier version of the present paper appeared, Bernd Schr\"oder \cite{bsch}  proved the endomorphism conjecture for posets of whidth at most 11. Though that result is stronger, the proof presented here is significantly shorter and simpler.

We construct endomorphisms from {\em each} automorphism of a poset, by shrinking two or more elements into 
one. For this approach to work, we will either need pairs elements $x$ and $y$  (``siblings") so that $x$ covers every element that $y$ covers, and $x$ is covered by every element that $y$ is covered by, or an element (a ``central element")  that is smaller than  every element of a given higher rank and larger than every element of a given lower rank.

\section{Tools}

Note that the following statement is equivalent to the endomorphism conjecture, even though it may look weaker at first sight. 
Let $P_1,P_2,\cdots$ be an infinite sequence of distinct finite posets. 
Denote $a_i=\frac{|\Auto(P_i)|}{|\Endo(P_i)|}$.  
Then 
\begin{equation} \label{notweaker}  
	\liminf a_i =0.
\end{equation}
If the endomorphism conjecture holds, then it is obvious that (\ref{notweaker}) also holds. 
On the other hand, if the endomorphism conjecture is false, then there is a  sequence $P_i$ of distinct finite posets for which there exists $\epsilon > 0$ so that for each element of an infinite subsequence $i_1 < i_2< \cdots $, the inequality $a_{i_j} \geq \epsilon$ holds. 
Therefore, in this case, $\liminf a_i \neq 0 $, and so (\ref{notweaker}) is false as well. 
We will make use of this observation in this paper several times; we will prove the endomorphism conjecture for various classes of posets by proving (\ref{notweaker}) for them. 

All posets in this paper are finite. 
A {\em graded poset} is a poset $P$ in which all maximal (nonextendible) chains have the same length.
If $P$ is graded,  then it has  rank function $r$ so that $r(y)=r(x)+1$ if $y$ covers $x$, with minimal elements of $P$ having rank 0. 
Let $w_i=w_i(P)$ denote the number of elements of rank $i$ in $P$. The numbers $w_i$ are called the {\em Whitney numbers}. 
For convenience, we define the {\bf whidth} of $P$ as the largest of the Whitney numbers $w_i(P)$. Note that as for all $i$, the elements of rank $i$ form an antichain in $P$,
 the whidth of $P$ is at most as large as the {\em width} of $P$, which is the size of the largest antichain of $P$.  For this reason, our results apply to a class of posets
that properly contains the class of all graded posets of width four. 

The rank of a graded poset $P$ is the rank of its maximal elements. 
For instance, a chain that consists of $n+1$ elements has rank $n$.  The height of a  poset is the number of elements in its longest chain. So if $P$ is graded,
its height is one larger than its rank.

As all posets in this paper have whidth at most four, we can assume without loss of generality that they have a connected Hasse
diagram. Indeed, a poset of whidth four does not have a connected Hasse diagram, then its Hasse diagram can have at most four components,
and so the size of its automorphism group is at most 24 times the product of the sizes of the automorphism groups of its components.

If $P$ is a graded poset and $T$ is a set of nonnegative integers, then the {\em rank-selected subposet} $P_T$ of $P$ is an induced subposet of $P$ that consists of all elements of $P$ whose ranks are in the set $T$. 
For instance, all two- and five-element subsets of the $n$-element set form the rank-selected subposet $P_{2,5}$ of the Boolean algebra $B_n$.  

For a positive integer $k$, let $W_k$ be the set of all finite graded posets that have whidth at most $k$. 
So the Hasse diagram of a poset in $W_k$ is similar to a tower. 
For instance, $W_1$ is the set of all chains.

We start by showing that in infinite sequences of posets in $W_k$, we will find an unbounded number of disjoint copies of some subposet. 

\begin{proposition} \label{repeating}
	Let $k$ and $m$ be  positive integers.
	Let $C$ be an infinite sequence of distinct posets in $W_k$. 

	Then there exists a graded poset $R$ of rank $m$, and an infinite subsequence $C'$ of $C$ so that the following holds for  $C'$. 
For a given element poset $P\in C'$, let $c(P)$ the largest positive integer so that $P$ contains $c(P)$ disjoint copies of $R$ as a rank-selected subposet defined by consecutive integers
as ranks. Then $c(P)$ goes to infinity as $P=P_1,P_2,\cdots$ ranges over $C'$.
\end{proposition}

If, in a sequence $C$ of posets,  a rank-selected subposet $R$ defined by consecutive ranks has the property described in Proposition~\ref{repeating}, then we will say that $R$ is a {\em repeating} subposet in $C$. 

\begin{proof}
	The size of the elements of $C$ goes to infinity. 
	Therefore, the total number of rank-selected subposets of the poset $P_i$ that consist of $m+1$ consecutive ranks also goes to infinity. 
	On the other hand, the number of possible posets of rank $m$ is bounded, as the number of elements of any given rank is bounded by $k$. 
	So by the pigeon-hole principle, one such poset must occur an unbounded number of times, which implies our claim. 
\end{proof}

Let us call an element $x\in P$ {\em up-single} if it is covered by exactly one element in $P$. 
Similarly, let us say that $x\in P$ is {\em down-single} if $x$ covers exactly one element in $P$. 
The next proposition shows that up-single and down-single elements are very useful for finding endomorphisms. 

\begin{proposition} \label{singles} 
	Let $k$ and $m$ be  positive integers.
	Let $C$ be an infinite sequence of distinct posets in $W_k$. 
	Let $R$ be a repeating subposet in an infinite subsequence $C'$ of $C$. 

	If $R$ contains an up-single or a down-single element, then 
	\[ \lim_{i\rightarrow \infty} \frac{|\Auto(P_i)|}{|\Endo(P_i)|} =0, \]
	where the $P_i$ are elements of the sequence $C'$. 
\end{proposition}

\begin{proof}
	Without loss of generality, let us assume that $R$ contains an up-single element.
 Let us choose a copy $R_1$ of $R$ in $P_i$, and let $x$ be an up-single element in $R_1$.
	Let $y$ be the only element of $P_i$ covering $x$. 
	We define the endomorphism $U_x$ of $P_i$ as follows.
	\begin{equation*}
		U_x(z) 	= 	\begin{cases*}
						z 	& if $z\neq x$ \\
						y 	& if $z=x$.
					\end{cases*}
	\end{equation*}

	Let $\phi \in \Auto(P_i)$. 
	Then $\phi_x=U_x\circ \phi \in\Endo(P_i)$. 

Now we show that $\phi_x\neq \phi_{x'}'$ unless $x=x'$ and $\phi=\phi'$.
	Indeed, if $x\neq x'$, then $\phi_x \neq \phi_{x'}'$ regardless of whether  $\phi$ and $\phi'$ are identical or not as $\Img(\phi_x) \neq \Img(\phi_{x'})$. 
	Indeed, the left-hand side does not contain $x$, the right-hand side does. 
	On the other hand, if $\phi\neq \phi'$, then $\Img(\phi)$ and $\Img(\phi')$ differ in the image of at least one element $z\neq x$, and therefore,
  $U_x\circ \phi  \neq U_x\circ \phi'$ because   $(U_x\circ \phi) (z)  \neq (U_x\circ \phi')(z)$.

	This means that if $P_i$ has $u$ copies of $R$, then $P_i$ has at least $u$ up-single elements that can play the role of $x$, implying that 
	\[ \frac{|\Auto(P_i)|}{|\Endo(P_i)|} \leq \frac{1}{u} . \] 
	As $u$ goes to infinity, this proves our claim. 
\end{proof}

Let $a$ and $b$ be two elements of $P$ that have the same rank. 
We say that $b$ is an {\em older sibling} of $a$ if $b$ covers all elements that $a$ covers, and $b$ is covered by all elements that cover $a$.  Note that this means it is possible that $b$ is an older sibling of $a$ and $a$ is an 
older sibling of $b$. If that happens, we will say that $a$ and $b$ are {\em twins}. 

The next proposition shows the usefulness of this definition. 

\begin{proposition} \label{siblings}
	Let $k$ and $m$ be  positive integers.
	Let $C$ be an infinite sequence of distinct posets in $W_k$. Let $R$ be a repeating subposet in an infinite subsequence $C'$ of $C$. 

	If $R$ contains a pair $(a,b)$ so that $b$ is an older sibling of $a$,  then  
	\[ \lim_{i\rightarrow \infty} \frac{|\Auto(P_i)|}{|\Endo(P_i)|} =0 , \]
	where the $P_i$ are elements of the sequence $C'$. 
\end{proposition}

\begin{proof}
	We define the endomorphism $V_{a,b}$ of $P_i$ as follows.
	\begin{equation*}
		V_{a,b}(z) 	= 	\begin{cases*}
						z 	& if $z\neq a$ \\
						b 	& if $z=a$.
					\end{cases*}
	\end{equation*}

	Let $\phi \in \Auto(P_i)$. 
	Then $\Phi_{a,b}=V_{a,b}\circ \phi \in \Endo(P_i)$.  
	If $(a,b)$ and $(a',b')$ are two pairs of  elements of $P_i$ so that $b$ is an older sibling of $a$ and $b'$ is an older sibling of $a'$, and $a\neq a'$,  then $\Phi_{a,b}\neq \Phi_{a',b'}$ as $\Img(\Phi_{a,b}) \neq
\Img(\Phi_{a',b'})$. 
	On the other hand, for a given pair $(a,b)$, the equality
	\begin{equation} \label{coincidence} 
		V_{a,b}\circ \phi = V_{a,b} \circ \phi' 
	\end{equation} 
	cannot hold for the automorphisms $\phi\neq \phi'$ of $P_i$ if there exists an element $z\in P_i$ so that $\phi(z)\neq \phi'(z)$ and $\phi(z)\notin \{a,b\}$. Indeed, if 
 $s=\phi(z)\notin \{a,b\}$, then $s$ is the {\em unique} preimage of $s=V_{a,b}(s)$ under $V_{a,b}$. In other words, $\{s\}=V_{a,b}^{-1}(s)$.

In other words, for (\ref{coincidence}) can hold only if $\phi(z)\neq \phi'(z)$ implies that $\{\phi(z),\phi'(z) \} = \{a,b\}$.  Therefore, for every $\phi \in \Auto(P_i)$, there is only one $\phi' \in \Auto(P_i)$ other than $\phi$ itself so that (\ref{coincidence}) holds for a given pair $\{a,b\}\in P_i$. Specifically, if $\phi(v)=a$ and $\phi(w)=b$, then $\phi \neq \phi'$ implies
that $\phi'(v)=b$ and $\phi'(w)=a$. Note that  this can only happen if $a$ an $b$ are twins.  


	Therefore, if there are $c$ disjoint copies of  $R$ in $P_i$, and $a$ and $b$ are not twins, then 
there are at least $c$ candidates for $a$, implying that there exist at least $c|\Auto(P_i)|$ endomorphisms of the form $V_{x,y} \circ \phi $. If $a$ and $b$ are twins, then 
there are at least $2c$ candidates for $a$, implying that the multiset of all endomorphisms of the form $V_{x,y} \circ \phi $ is of size at least $2c|\Auto(P_i)|$, and none of its elements have multiplicity more than two, therefore
	\[ \frac{|\Auto(P_i)|}{|\Endo(P_i)|} \leq \frac{2}{2c}=\frac{1}{c} . \] 
	As $c$ goes to infinity, this proves our claim. 
\end{proof}



\section{whidth 2}
The tools developed in the last section make it easy to prove the endomorphism conjecture for posets of whidth two. 

\begin{theorem} \label{whidth2} 
	Every infinite sequence of distinct posets in $W_2$ contains an infinite subsequence $P_1,P_2,\ldots $ of posets so that  
	\[ \lim_{i\rightarrow \infty} \frac{|\Auto(P_i)|}{|\Endo(P_i)|} =0 . \]
\end{theorem}

\begin{proof} 
	By Proposition~\ref{repeating}, our sequence contains an infinite subsequence $P_1,P_2,\ldots $ of posets that contain an unbounded number of disjoint copies of a repeating poset $R$ of rank 2. 
	If $R$ contains an up-single or down-single element, then our claim follows from Proposition~\ref{singles}. 
	If not, then the two elements of rank 1 in $R$ are twins, and our claim follows from Proposition~\ref{siblings}. 
\end{proof}

\section{whidth 3}

\begin{theorem}  \label{whidth3} 
	Every infinite sequence of distinct posets in $W_3$ contains an infinite subsequence $P_1,P_2,\ldots $ of posets so that  
	\[ \lim_{i\rightarrow \infty} \frac{|\Auto(P_i)|}{|\Endo(P_i)|} =0 . \]
\end{theorem}

\begin{proof}
	We can assume that no repeating subposet in the $P_i$ contains an up-single or a down-single element; otherwise, our claim immediately follows from Proposition~\ref{singles}. 

	Now let $s$ and $t$ be two elements in a repeating subposet $R$ of the sequence of the $P_i$ so that $r(s)+2=r(t)$. 
	Then we claim that $s<t$.
	Indeed, $s$ is covered by at least two elements, $t$ covers at least two elements, but there are at most three elements of rank $r(s)+1$ in $P_i$, so there has to be at least one element $z$ that covers $s$ and is covered by $t$.

	Let $R$ be a repeating subposet of rank four for the sequence $P_1,P_2,\ldots $. Let $x$ be an element at the middle level in a copy $R_1$ of $R$ in $P_i$. 
	That is, $r(x)=2$ in $R_1$. 
	Then by the previous paragraph, for all minimal elements $a$ and all maximal elements $b$ of $R_1$, the chain of inequalities $a<x<b$ holds. 
	Now we define $F_x\in \Endo(P_i)$ as follows.
	Let $R_{1,int}$ be the poset $R_1$ with its minimal and maximal elements removed. 
	Let
	\begin{equation*}
		F_x(z) 	= 	\begin{cases*}
						z 	& if $z\notin R_{1,int}$ \\
						x 	& if $z\in R_{1,int}$.
					\end{cases*}
	\end{equation*}

	In other words, $F_x$ contracts $R_{1,int}$ to the point $x$, and leaves all other points fixed. 
	It follows that for all $\phi \in \Auto(P_i)$, the map $F_x \circ \phi $ is an endomorphism of $P_i$. 
	If $x$ and $x'$ are different elements in $P_i$, then  
	\[ F_x \circ \phi \neq F_{x'} \circ \phi' \]
	regardless of whether $\phi$ and $\phi'$ are identical or not. Indeed for the endomorphisms on both sides, there is one element that has more than one preimage, 
but that element is not the same on the two sides; it is $x$ on the left-hand side and $x'$ on the right-hand side.
	For a given $x$, the equality $F_x \circ \phi =F_x \circ \phi '$ can hold only if $\phi$ and $\phi'$ differ only in an automorphism of $R_{1,int}$. 
	As $R_{1,int}$ is a poset of rank two and whidth three, it has at most $3!^3=216$ automorphisms. 

	Therefore, if there are $c$ copies of $R$ in $P_i$, then the inequality 
	\[ \frac{|\Auto(P_i)|}{|\Endo(P_i)|} \leq \frac{216}{c} \] 
	holds. 
	Indeed, each automorphism of $P_i$ gives rise to $c$ endomorphisms, and no endomorphism will be obtained more than 216 times in this way. 
	As $c$ goes to infinity, this implies that $\frac{|\Auto(P_i)|}{|\Endo(P_i)|}$ goes to 0 as claimed. 
\end{proof}

\section{whidth 4}
\begin{theorem}
	Every infinite sequence of distinct posets in $W_4$ contains an infinite subsequence $P_1,P_2,\ldots $ of posets so that  
	\[ \lim_{i\rightarrow \infty} \frac{|\Auto(P_i)|}{|\Endo(P_i)|} =0 . \]
\end{theorem}

\begin{proof}
Throughout this proof, keep in mind that by Proposition \ref{repeating}, for any positive integer $m$, a repeating rank-selected subposet $R$ of rank at least $m$ exists.

	Let $R$ be a repeating subposet in our sequence.
	\begin{itemize}
		\item 	Note that $R$ cannot contain an up-single or a down-single element, otherwise the proof follows from Proposition~\ref{singles}. 
		\item 	Also note that $R$ cannot contain vertices $a$ and $b$ so that $b$ is an older sibling of $a$, otherwise the proof follows from Proposition~\ref{siblings}. In particular, $R$ cannot contain vertices $a$ and $b$ that are twins. 
	\end{itemize}  
	We can assume that our sequence does not contain such repeating subposets. 

	Now let us assume that our sequence contains a repeating subposet $R$ that contains an element $x$ so that $a<x<b$ for every element $a\in R$ for which $r(x)-r(a)=r_1$, and for every $b\in R$ for which $r(b)-r(x)=r_2$, for some fixed integers $r_1$ and $r_2$. 
	In other words, $x$ is larger than everything that is $r_1$ levels below it, and smaller than everything that is $r_2$ levels above it. 
	Let us call such a repeating subposet a {\em good} repeating subposet, and let us call an element $x$ with the described property a {\em central} element.  
	Note that this phenomenon was discussed in the proof of Theorem~\ref{whidth3}, in the special case of $r_1=r_2=2$. 
	In this case, the proof can be completed as the proof of Theorem~\ref{whidth3}, namely by defining endomorphisms $F_x$ that shrink the entire interior of such an $R$ to $x$ and leave everything else fixed. 

	Can we always find a good repeating subposet? 
	If $x$ is an element that is smaller than three elements of rank $r$, then $x$ is smaller than all four elements of rank $r+1$. 
	Indeed, suppose $y$ is a rank $r+1$ element such that $x\not< y$.
	As $y$ is not down-single, it is larger than at least two elements of rank $r$, and the set of those two elements cannot be disjoint from the set of rank $r$ elements that are larger than $x$. 
	Similarly, if $x$ is larger than three elements of rank $s$, then $x$ is larger than all  elements of rank $s-1$. 

	Therefore, if our sequence contains a repeating subposet $R$ so that there is an element $x\in R$ that is larger than three elements of $R$ having rank $r(x)-r_1$, and is smaller than three elements of $R$ having rank $r(x)+r_2$, and $R$ has some elements of rank less than $r(x)-r_1$, and $R$ has some elements of rank more than $r(x)+r_2$, then $R$ is a good repeating subposet and $x$ is a central element in $R$. 

	Since no repeating poset $R$ contains up-single or down-single elements, the only remaining cases are when each element in repeating posets has at least one of the following two properties: (a) it is larger than exactly two elements of each rank below it or (b) it is smaller than exactly two elements of each rank above it. 

	Let $x$ and $y$ be two  elements of the same rank $r$ in a repeating subposet $R$. We can assume that two such elements exist; otherwise, if $x$ is the only element of rank $r\in R$, then all of $R$ can be shrunk to 
$x$
by an endomorphism. As the number of copies of $R$ goes to infinity, it will eventually be larger than any bound $B$. That implies that every automorphism in our poset can be turned into more than $B$ endomorphisms by shrinking a copy of $R$ to $x$. At least $B/|\Auto(R)|$ of these automorphisms will be different.    Similarly, we can assume that two such elements exist so that either both belong to case (a) or both belong to case (b) in the previous paragraph; otherwise, there could be only two elements of rank $r\in R$.
(If that holds for rank $r-1$ and rank $r+1$ as well, then the two elements of rank $r$ are twins.)  Let us assume without loss of generality that each of $x$ and $y$ has  the property that it is smaller than exactly two elements of each rank above them in $R$. 
	For convenience, let $r$ be the smallest rank for which such elements $x$ and $y$ exist.
	This means in particular that there are two elements of rank $r+1$ that are larger than $x$, and two elements of rank $r+1$ that are larger than $y$. 
	So, if $S$ is the set of all elements of rank $r+1$ that are larger than at least one of $x$ and $y$, then $S$ can have two, three, or four elements.
	We will consider these cases separately.  

	\begin{enumerate}
		\item 	If $|S|=4$, then let  $\{x_{1,1},x_{1,2}\}$ denote the elements covering $x$, and let  $\{y_{1,1},y_{1,2}\}$ denote the elements covering $y$.
 The sets $\{y_{1,1},y_{1,2}\}$ and  $\{y_{1,1},y_{1,2}\}$  are disjoint, since we are in the case $|S|=4$. 
				Note that the 2-element set of vertices covering $x_{1,1}$ has to be the same as the 2-element set of vertices covering $x_{1,2}$; otherwise $x$ would be smaller than at least three elements of rank $r+2$. 
				Call these elements $x_{2,1}$ and $x_{2,2}$. Analogously, define $x_{3,1}$ and $x_{3,2}$.

				Similarly, the two vertices covering $y_{1,1}$ must be the same as the two vertices covering $y_{1,2}$. 
				Call these vertices $y_{2,1}$ and $y_{2,2}$.  Analogously, define $x_{3,1}$ and $x_{3,2}$.

				We are going to distinguish three cases, based on the size of the intersection $\{x_{2,1},x_{2,2}\} \cap \{y_{2,1},y_{2,2}\}$. 

			\begin{itemize} \item Let us first assume that $\{x_{2,1},x_{2,2}\}$ and  $\{y_{2,1},y_{2,2}\}$ are disjoint. 	Note that $x_{2,1}$ and $x_{2,2}$ are twins. 
				Indeed, they both cover $x_{1,1}$ and $x_{1,2}$, and they cannot cover $y_{1,1}$ or $y_{1,2}$, since that would mean that those vertices are covered by more than the two elements $y_{2,1}$ and $y_{2,2}$, and so $y$ is covered by more than two elements of rank $r+2$. 
				Furthermore, $x_{2,1}$ and $x_{2,2}$ are both covered by $x_{3,1}$ and $x_{3,2}$ and nothing else.  However, we assumed that there are no twins in $R$, so we reached a contradiction.
\item Let us now assume that $|\{x_{2,1},x_{2,2}\} \cap \{y_{2,1},y_{2,2}\}|=1$. We can assume without loss of generality that $x_{2,2}=y_{2,1}=z$ is the unique element of that 
intersection. Then it follows from our
definitions that $z$ is larger than all four elements of $P$ that are of rank $r+1$, and it is smaller than both elements of $r+3$, namely $x_{3,1}$ and $x_{3,2}$, which are the same
as $y_{3,1}$ and $y_{3,2}$. Indeed, there cannot be any additional elements $s$ of rank $r+3$, since such an $s$ would have to be larger than an element $s'$ of rank $r+2$, implying
that $x$ or $y$ is less than $s'$, contradicting our assumptions. So $z$ is a central element in $R$.
 \item  Let us finally assume that $\{x_{2,1},x_{2,2}\}$ and  $\{y_{2,1},y_{2,2}\}$ are identical. 
Then each of $x_{2,1}$ and $x_{2,2}$ cover all four elements of rank $r+1$ in $R$.  If the upper neighbors of $x_{2,1}$ and $x_{2,2}$ are not identical, then together they have at least three upper neighbors, meaning that $x$ (and $y$) are smaller than three elements of rank $r+3$, which contradicts our hypothesis.  If the upper neighbors of $x_{2,1}$ and $x_{2,2}$ are identical, then $x_{2,1}$ and $x_{2,2}$ are twins. This contradicts our assumption that there are no twins in $R$. 
\end{itemize}

		\item 	If $|S|=3$, then let $x_1$ be the element that covers $x$, but not $y$, let $y_1$ be the element that covers $y$, but not $y$, and let $z$ be the element that covers both $x$ and $y$. 
				Then $z$ must be covered by exactly two elements, otherwise $x$ and $y$ are less than three or more elements of rank $r+2$. 
				Let $z_1$ and $z_2$ be the two elements covering $z$; then $z_1$ and $z_2$ must also be the two elements covering $x_1$, and the two elements covering $y_1$ (otherwise again, $x$ and $y$ are smaller than three or more elements of rank $r+2$). 
				However, this leads to a contradiction, since if $h$ is another element of rank $r+2$, then $h$ must cover at least two elements of rank $r+1$, so it must cover at least one of $z$, $x_1$, and $x_2$, implying that either $x$ or $y$ is less than three vertices of rank $r+2$.  

		\item 	Finally, suppose that  $|S|=2$ and moreover, $R$ has no rank that has two vertices for which the set of elements that cover at least one of them is of size $3$ or $4$.  
				Denote the elements of $S$ by $x_1$ and $y_1$. 
				As each of $x$ and $y$ has exactly two elements of rank $r+2$ larger than it, it follows that $x_1$ and $y_1$ are smaller than precisely those two elements of rank $r+2$, which we call $x_2$ and $y_2$. 
				Repeating this process, we can define $x_j$ and $y_j$ all the way to the top of $R$.

				For each $i\geq r+1$, let $z_i$ and $t_i$ be the two other elements that share the rank of $x_i$ and $y_i$. (If, for some $i$, there are no elements other
than $x_i$ and $y_i$ that are of rank $i$, then the elements $x_{i+1}$ and $y_{i+1}$ are twins. If there is an $i$ so that there is only one additional element $z_i$ of rank $i$, then element
$z_{i+1}$ has to cover at least two elements of rank $i$, so it is forced to cover $x_i$ or $y_i$, which is a contradiction.)
				For each $i$, the elements $z_i$ and $t_i$ have to cover at least two elements, but they cannot cover $x_{i-1}$ and $y_{i-1}$ (because then two elements would be covered by more than two elements, putting us in one of the previous cases), so they must cover $z_{i-1}$ and $t_{i-1}$. 

				Consider $x_2$ and $y_2$. 
				If they only cover $x_1$ and $y_1$, then they are twins, which contradicts our assumptions. 
				If the set of elements covered by one is a subset of the elements covered by the other, then one of them is an older sibling of the other, which again contradicts our assumptions. 
				So it has to be that $x_2$ and $y_2$ each cover one additional element, and we can assume without loss of generality that $x_2$ covers $z_1$ and $y_2$ covers $t_1$. Then $z_1$ and $t_1$ must each cover two elements of rank $r$, and they cannot cover $x$ and $y$. Therefore, there must be two additional elements $z_0$ and $t_0$ of rank $r$, and each of $z_1$ and $t_1$ must cover both of them. 

				Then the map $f$ defined by $f(x_2)=y_2$, $f(z_1)=t_1$, and $f(h)=h$ for all other elements of $P_i$ is an endomorphism of $P_i$. 
				The proof is now finished \`a la the proof of Theorem~\ref{whidth3}. That is, if $\phi \in \Auto( P_i)$, then $f\circ \phi \in \Endo(P_i)$, for $f$ defined as above using
a given copy of $R$, the number of {\em different} endomorphisms of the form  $f\circ \phi $ is at least half $|\Auto( P_i)|$, but the number of choices for the copy of $R$ that we use
goes to infinity. 
	\end{enumerate}
\end{proof}

\section{Conclusion}

Schr\"oder~\cite{schroder-1} remarks that it is not known if $\frac{|\Auto(P)|}{|\Endo(P)|}\leq 1/2$ for all posets $P$ on at least two elements. The arguments from~\cite{bona} establish that this is true for posets that have a zero. Moreover,  it follows that $\frac{|\Auto(P)|}{|\Endo(P)|}\leq 1/3$ for all posets $P$ on at least three elements that have a zero. To see this, let the set of atoms be $A$, where an atom is an element that covers zero. If $|A|\geq 2$, then for every automorphism, one can compose that automorphism with mapping at most one of the atoms to zero. If $|A|=1$, then one can compose every automorphism with mapping that atom or its parent to zero. In either case, this results in three times as many endomorphisms as automorphisms. 

We note that the case where the sequence of posets consists of those that are of height-2 (that is, rank-1) is also unresolved for the endomorphism conjecture.

\section{Acknowledgements}

We would like to thank anonymous referees for insightful and important comments, including the observation that it is worth noting that the above results held for graded posets whose largest Whitney number is at most 4, which resulted in coining the term ``whidth''.


\begin{thebibliography}{99}
	\bibitem{bona} 			M. B\'ona, On the endomorphism conjecture for posets with 0. {\it Order} {\bf 14} (1997/98), no. 3, 191--192. 
	\bibitem{kuzjurin} 		N. N. Kuzjurin, On the automorphism conjecture for products of ordered sets, {\it Order} {\bf 9} (1992), no. 3, 205--208. 
	\bibitem{liu-wan} 		W.-P. Liu and H. H. Wan, Automorphisms and isotone self-maps of ordered sets with top and bottom,  {\it Order} {\bf 10} (1993), no. 2, 105--110.  
	\bibitem{rival} 			I. Rival and A. Rutkowski, Does almost every isotone self-map have a fixed point? in {\it Extremal Problems for Finite Sets (Visegr\'ad, 1991)}, 413--422, Bolyai Soc. Math. Stud., 3, {\it J\'anos Bolyai Math. Soc., Budapest}, 1994. 
	\bibitem{schroder} 		B. S. W. Schr\"oder, The automorphism conjecture for small sets and series parallel sets,  {\it Order} {\bf 22} (2005), no. 4, 371--387. 
	\bibitem{schroder-1}  	B. S. W. Schr\"oder, The automorphism conjecture for ordered sets of dimension 2 and Interval Orders,  {\it Order} {\bf 38} (2021), no. 2, 271--281. 
\bibitem{bsch} B. S. W. Schr\"oder,  The Automorphism Conjecture for Ordered Sets of width  $\leq 11$. Preprint, 
available at \hbox{https://arxiv.org/abs/2209.09312}. 
\end{thebibliography}
\end{document}